\documentclass{amsart}

\usepackage[utf8]{inputenc}
\usepackage[english]{babel}
\usepackage{mathabx}
\usepackage{amsthm}
\usepackage[shortlabels]{enumitem}
\usepackage{amssymb}
\usepackage{amsfonts}
\usepackage{accents}
\usepackage[dvipsnames]{xcolor}
\usepackage{mathtools}
\usepackage{hyperref}
\usepackage{bbm}
\usepackage{amsmath}

\theoremstyle{plain}
\newtheorem{prop}{Proposition}[section]

\newtheorem{teo}[prop]{Theorem}

\newtheorem{lem}[prop]{Lemma}
\newtheorem{cor}[prop]{Corollary}
\newtheorem{theorem}{Theorem}

\newtheorem{corollary}[theorem]{Corollary}

\theoremstyle{definition}

\newtheorem{set}[prop]{Setup}

\newtheorem{rk}[prop]{Remark}

\newtheorem{exa}[prop]{Example}

\newtheorem{quest}[prop]{Question}
\newtheorem{df}[prop]{Definition}

\theoremstyle{remark}

\newcommand{\F}{\mathcal{F}}

\newcommand{\N}{\mathbb{N}}

\newcommand{\R}{\mathbb{R}}
\newcommand{\Z}{\mathbb{Z}}
\newcommand{\G}{\Gamma}

\newcommand{\scl}{\text{\normalfont{scl}}}
\newcommand{\cl}{\text{\normalfont{cl}}}
\newcommand{\X}[1]{\(\mathfrak{X}^{\textup{#1sep}}\)}

\title[Bounded cohomology and scl of verbal wreath products]{bounded cohomology and scl of\\ verbal wreath products}

\author{Elena Bogliolo}
\date{\today.\text{ The results in this paper are part of the Ph.D. thesis of the author}}

\begin{document}

\begin{abstract}
We study the bounded cohomology and the stable commutator length of verbal wreath
products $\Gamma \wr^{_W}A$, where $A$ has trivial
bounded cohomology for a sufficiently large class of coefficients.\\
We prove that the stable commutator length always vanishes, and that the bounded cohomology vanishes in positive degrees for some such verbal wreath products; including the standard restricted wreath products (extending a recent result by Monod for lamplighters groups), as well as verbal wreath products arising from n-solvable, $n$-nilpotent, and $k$-Burnside $(k = 2, 3, 4, 6)$ verbal products.\ As an application, we show that every group of type $F_p$ isometrically embeds into a group of type $F_p$ with vanishing bounded cohomology in positive degrees for a large class of coefficients. 
\end{abstract}

\maketitle

\section*{Introduction}

Bounded cohomology of groups encodes many of their algebraic and geometric properties.\ For instance, a celebrated theorem by Johnson \cite{joh} shows that the vanishing of bounded cohomology  of a group $\Gamma$ in all positive degrees, for all \emph{dual} Banach $\Gamma$-modules~$V$, $H_b^{n \geq 1}(\Gamma; V)=0$, is equivalent to the amenability of $\Gamma$. 

A recent breakthrough result by Monod shows that Thompson group $F$ has vanishing bounded cohomology in positive degrees for all separable coefficient modules~\cite{lamp}.\ 
Of course, this class of coefficients is smaller than the one of dual Banach coefficients, but still the previous result brought renewed interest in the understanding of the classes of coefficients for which the vanishing of bounded cohomology holds.\ 
A fundamental result following from Monod's construction is the proof that all restricted permutational wreath products \[\Gamma \wr_X A \coloneqq \big(\bigoplus_{x \in X} \Gamma\big) \rtimes A\] have vanishing bounded cohomology in positive degrees for all separable coefficients provided that $A$ is amenable and the action of $A$ on $X$ has infinite orbits \cite{criterion}.

Given a class \(\mathfrak{X}\) of coefficient modules, we say that a group \(\G\) is \emph{boundedly \(\mathfrak{X}\)-acyclic} if its bounded cohomology vanishes in all positive degrees for every module in~\(\mathfrak{X}\).\ In particular we denote by \X{} the class of separable coefficient modules.\
As we mentioned above, \(\mathfrak{X}^{\textup{dual}}\)-boundedly acyclic groups are amenable groups; moreover, \(\mathfrak{X}^{\textup{all}}\)-boundedly acyclic groups are finite groups \cite[Theorem 3.12]{fri}.\ Monod showed that there exist non-amenable groups that are \X{}-boundedly acyclic \cite[Theorem 3 and Section 4.3]{lamp}.\ We refer the readers to groups with commuting cyclic conjugates \cite{criterion} for other instances of \X{}-boundedly acyclic groups that are not amenable.

In this paper we aim to extend Monod's result on restricted permutational wreath products by removing the constraint on $A$ to be amenable and replacing this requirement with the weaker one of being boundedly acyclic for a large class of coefficients.\ 
This result will enable us to prove the bounded acyclicity of certain iterated wreath products.\ 
One of the main problems in extending Monod's construction, leaving the comfortable setting of \emph{amenable groups}, is that we cannot use the theory of Zimmer-amenable actions in its full power anymore \cite[Chapter 5]{Mon}.\ 
We solve this issue by considering restricted permutational wreath products $\Gamma \wr_X A$ where $A$ is not only boundedly acyclic but has vanishing bounded cohomology in positive degrees for all \emph{semi-separable} coefficients.\ 
This requirement is sufficient to enable us to deal with spectral sequences and acyclic resolutions instead of using the classical theory by Burger and Monod for computing bounded cohomology \cite{BurMon}.\medskip
  
We say that a coefficient \(\G\)-module \(V\) is \emph{semi-separable} if there exists a separable coefficient \(\G\)-module \(U\) and an injective dual \(\G\)-morphism \(i\colon V\to U\).\ Let \X{s} denote the class of all semi-separable coefficient \(\G\)-modules.\
The class of \X{s}-boundedly acyclic groups is knowns to be larger that the class of amenable groups.\ On the other hand we point out that, to this day, there are no known examples of groups that are \(\R\)-boundedly acyclic but not \X{s}-boundedly acyclic (\cite[Problem 4.1]{disp} and  Question \ref{quest}).\

We prove the following:

\begin{theorem}[Theorem \ref{acyclicwr}]\label{A}
    Let \(\Gamma\) and \(A\) be discrete groups.\ Let \(n\geq 1\) and suppose that \(A\) is \X{s}-boundedly \(n\)-acyclic and countable.\
    If \(X\) is a countable \(A\)-set such that all the \(A\)-orbits are infinite, then, 
    the restricted permutational wreath product \(\Gamma\wr_X A\) is \X{s}-boundedly \(n\)-acyclic.
\end{theorem}

We emphasize that our approach, in contrast to previous results, allows for degree-wise statements, as stated above.\
This leads to many new examples of boundedly acyclic groups, as we show in Section \ref{5}.\medskip

In recent years there has been an increasing interest in \emph{verbal wreath products}  (denoted by \(\G\wr_X^{_W}A\)) \cite{Brude} and their generalizations, e.g.\ halo products \cite{Halo}.\
In particular Brude and Sasyk studied the following families of restricted verbal wreath products:
\begin{enumerate}
    \item \(j\)-nilpotent for \(j\in \N\);
    \item \(j\)-solvable for \(j\in \N\);
    \item \(k\)-Burnside for \(k\in\{2,3,4,6\}\).
\end{enumerate}

Here we prove that in the cases listed above the restricted verbal wreath product with \X{s}-boundedly \(n\)-acyclic groups is \X{s}-boundedly \(n\)-acyclic (see Section \ref{6} for the definition of verbal wreath products):

\begin{corollary}[Corollary \ref{verbalcase}]\label{B}
    Let \(A\) be a countable \X{s}-boundedly \(n\)-acyclic group and let \(X\) be a countable \(A\)-set with infinite \(A\)-orbits.\ Then, for any group \(\G\) and every \(j\in \N\) and \(k\in \{2,3,4,6\}\), the restricted \(j\)-nilpotent, \(j\)-solvable, and \(k\)-Burnside wreath product \(\G\wr_X^{_W} A\) is \X{s}-boundedly \(n\)-acyclic.
\end{corollary}

\textbf{Applications:}
In the last two sections of this article, we give two applications of Theorem \ref{A}. The first one concerns stable commutator length of restricted verbal wreath products.\ Stable commutator length is an invariant of groups that is related to second bounded cohomology through Bavard’s Duality Theorem \cite{ Bavard,Cal}.\ We will use the results above, together with a result by Calegari on the stable commutator length of groups admitting a law \cite{cal2}, to show the vanishing of stable commutator length of \emph{any} verbal wreath product by \X{s}-boundedly acyclic groups:

\begin{theorem}[Theorem \ref{sclcase}]\label{C}
    Let \(A\) be a countable group, and let \(X\) be a countable \(A\)-set such that all the \(A\)-orbits are infinite.\
    If \(A\) is \X{s}-boundedly \(2\)-acyclic, then for any group \(\G\) the stable commutator length of any permutational verbal wreath product \(\G\wr_X^{_W} A\) vanishes. 
\end{theorem}

The second application allows to construct an embedding of any group \(\G\) into a \X{s}-boundedly acyclic group that preserves the finiteness properties of \(\G\).\ This extends a recent embedding result by Wu, Wu, Zhao and Zhou who proved the analogue for \(\R\)-boundedly acyclic groups \cite[Theorem 0.1]{embed2}.

\begin{theorem}[Theorem \ref{finitenesscase}]\label{D}
    Let \(\G\) be a group of type \(FP_m\) (respectively \(F_m\)).\ Then \(\G\) injects isometrically into an \X{s}-boundedly acyclic group of type \(FP_m\) (respectively \(F_m\)).
\end{theorem}

 \textbf{Plan of the paper:} We start by laying out the necessary background on
coefficient modules and bounded cohomology.\ In Sections \ref{2} and \ref{3} we look into the class of semi-separable modules and recall results on Zimmer-amenable and ergodic group actions.\ Section \ref{4} is dedicated to the proof Theorem \ref{A} that is then extended in Section \ref{6} to the case of verbal wreath products (Corollary \ref{B}).\ In Section \ref{5} we present some examples of applications of Theorem \ref{A}. Finally, in Sections \ref{7} and \ref{8} we discuss two applications regarding stable commutator length of verbal wreath products (Theorem \ref{C}) and isometric embeddings into boundedly acyclic groups respectively (Theorem \ref{D}).\medskip

\textbf{Acknowledgements:}
I am very grateful to my advisor Marco Moraschini for his guidance throughout this project and to Francesco Fournier-Facio for the ideas and suggestions he shared with me.\ I would like to thank Xiaolei Wu for bringing the Houghton groups to my attention.\ This work also benefited from conversations with Giuseppe Bargagnati, Caterina Campagnolo, Alberto Casali, Giorgio Mangioni, Filippo Sarti and the topology group in Pisa.

\section{Preliminaries}
Unless otherwise stated, every group in this article is assumed to be \emph{discrete}.\
In this section we recall definitions and preliminary results that we will need later.

\subsection{Coefficient modules}

 The notion of \emph{coefficient modules} was introduced by Monod as a convenient class of modules in the study of bounded cohomology \cite[Definition 1.2.1]{Mon}.\ We recall here the definition and we prove that under some closedness assumptions, submodules and quotients of coefficient modules are still coefficient modules.
\medskip\\
    As usual a map \(\phi\colon V\to U\)
between normed vector spaces is \emph{bounded} when it has finite operator norm
\[||\phi||_{\textup{op}} \coloneqq \sup_{x\in V\setminus \{0\}}
\frac{||\phi(x)||}{||x||}.\] 

\begin{df}[normed \(\G\)-module]
     Let \(\G\) be a group.\
A \emph{normed \(\G\)-module} is a normed \(\R\)-module \(V\) together
with a \(\G\)-action by linear isometric isomorphisms.\
When \(V\) is complete we say that \(V\) is a \emph{Banach} \(\G\)-module.\\
A morphism of normed \(\G\)-modules, simply called \emph{\(\G\)-morphism}, is a \(\G\)-equivariant bounded linear
map between normed \(\G\)-modules.
\end{df}

\begin{df}[coefficient \(\G\)-module {\cite[Definition 1.2.1]{Mon}}]
    Let \(\G\) be a group. A \emph{dual \(\G\)-module} \(V\) is the topological  dual of a Banach \(\G\)-module \(V^b\) endowed with the operator norm and the \(\G\)-action induced by the action of \(\G\) on \(V^b\).\ That is for every \(f\in V\) and \(g\in \G\), the action of \(g\) on \(f\) is defined as \((g\cdot f)(u)\coloneqq f(g^{-1}\cdot u)\) for every \(u\in V^b\).\\ A \emph{coefficient \(\G\)-module} is a dual \(\G\)-module \(V\) such that the predual \(V^b\) is separable.\\
    A \emph{dual \(\G\)-morphism} between dual \(\G\)-modules is a \(\G\)-morphism that is dual to some morphism between the preduals. 
\end{df}

\begin{df}[separable module]
    Let \(\G\) be a group.\ A coefficient \(\G\)-module \(V\) is \emph{separable} if there exists a countable dense subset in \(V\) with respect to the topology induced by the norm.\ We denote by \X{} the class of all separable coefficient \(\G\)-modules.
\end{df}

For every group \(\G\) the real line \(\R\) endowed with the trivial \(\G\)-action is a separable Banach module.\ Given a standard probability space \((\Omega,\mu)\) with a measurable, measure preserving \(\G\)-action, the space \(L^2(\Omega;\R)\) is also a separable Banach \(\G\)-module.\ We will see later an important example of coefficient \(\G\)-module that is not separable (Example \ref{Linfcoef}).

\begin{rk}\label{dual}
    A morphism of coefficient \(\G\)-modules \(\phi\colon V\to U\) is dual if and only if it is \emph{weak}-\(^*\) continuous \cite[Definition 1.2.1]{Mon}.
\end{rk}

\begin{rk}
    Observe that the induced action on a dual \(\G\)-module \(V\) is still by isometries.
\end{rk}

 As usual, when \(H\) is a subgroup of a discrete group \(\G\) and \(V\) is a coefficient \(\G\)-module, we denote the \emph{submodule of \(H\)-invariants} of \(V\) as:
\[V^H\coloneqq\{v\in V\vert\; h\cdot v=v \text{ for all } h\in H\}.\]

\begin{df}\label{quotmod}
    Let \(H\) be a subgroup of a group \(\G\) and let \(\iota:H\hookrightarrow G\) be the inclusion.\ Suppose that \(V\) is a coefficient \(\G\)-module.\ Then, the action of \(\G\) on \(V\) induces a \emph{restriction} action of \(H\) on \(V\) defined by \begin{eqnarray*}
   \phi_H\colon  H\times V&\to& V\\
   (h,v)&\mapsto &\iota(h)\cdot v.
\end{eqnarray*} If \(H\) is normal, let \(V^H\) be the submodule of \(H\)-invariants in \(V\).\ The action of \(\G\) on \(V^H\) passes to the quotient \(Q\cong {\G}/{\iota(H)}\) and induces an action of \(Q\) on \(V^H\) defined by
\begin{eqnarray*}
   \phi_Q\colon Q\times V^H&\to& V^H\\
   (gH,v)&\mapsto &g\cdot v.
\end{eqnarray*} 
\end{df}
We recall the following: 
\begin{lem}\label{closed}
    Let \(H\) be a subgroup of a discrete group \(\G\) and let \(V\) be a coefficient \(\G\)-module.\ Then the subspace \(V^{H}\) is closed and weak-* closed in \(V\).
\end{lem}
\begin{proof}
    Since \(\G\) acts by linear isometries, for every \(g\in \G\) the map \(\phi_g\colon V\to V\) that sends \(v\) to \(gv-v\) is continuous and weak-* continuous.\ Then \(V^H=\bigcap_{g\in H}\phi_g^{-1}(0)\) is closed as it is an intersection of closed subspaces.
\end{proof}

     Given a Banach space \(X\) and its \emph{topological dual} \(X'\), we denote by \(M^{\perp}\) the \emph{annihilator} of a subspace \(M\) of \(X\) (i.e.\ the set \(\{f\in X' \; \vert\; \langle f,x\rangle=0 \text{ for all } x\in M\}\)).
     Similarly for a subspace \(N\) of \(X'\) we denote by \(^{\perp}N\) its \emph{annihilator} (i.e.\ the set \(\{x\in X\; \vert\; \langle f,x\rangle=0 \text{ for all } f\in N\}\)).\ 
     Recall the following classical result of functional analysis:

\begin{teo}[{\cite[Theorem 4.7 and 4.9]{rudin}}]\label{rudindual}
Let \(M\) be a norm-closed subspace of a Banach space \(X\).\ The following isometric isomorphisms hold:
\begin{enumerate}
    \item  \( X'/M^{\perp}\cong M',\)
    \item \( M^{\perp} \cong(X/M)'.\)
    \item \((^{\perp}N)^{\perp}\) is the weak-* closure of \(N\) in \(X'.\)
\end{enumerate}
\end{teo}

Using the previous result we show that weak-* closed submodules and quotients of coefficient modules are dual.\ The following result is classical \cite[Section 1.2]{Mon} but we give a complete proof for convenience of the reader.

\begin{lem}\label{G-mod}
    Let \(V\) be a coefficient \(\G\)-module, and let \(W\) be a weak-* closed \(\G\)-invariant submodule of \(V\).\ Then \(W\) and \(V/W\) have an induced coefficient \(\G\)-module structure. 
\end{lem}
\begin{proof}
    Since \(^{\perp}W\) is the intersection of the kernels of the functionals in \(W\), the space \(^{\perp}W\) is closed in the predual \(V^b\) of \(V\) with respect to the Banach norm.\ Recall that, by definition, the predual of a coefficient module is required to be separable.\ Hence we have that also \(^{\perp}W\) and \(V^b/^{\perp}W\) are separable Banach modules as closed subsets and quotients by closed subsets of separable Banach modules.\ By the duality principle of Theorem
    \ref{rudindual} with \(M= {^{\perp}W}\) and \(X=V^b\), we obtain
    \[W\cong \Big(\frac{V^b}{^{\perp}W}\Big)'\quad\text{ and } \quad\frac{V}{W}\cong (^{\perp}W)'.\]
    This shows that \(W\) and \(V/W\) are Banach modules, dual to separable \(\G\)-modules.\ We are left to show that the \(\G\)-action on \(W\) and \(V/W\) induced by the \(\G\)-action on \(V\) coincides with the contragradient action.\ 
    The action of \(\G\) on \(V\) is the contragradient action to that of \(\G\) on \(V^b\) i.e.\ for all \(f\in V\) and \(g\in \G\), for every \(u\in V^b\), we have~\((g\cdot f)(u)=f(g^{-1}\cdot u)\).\\ Let \(f\in W\subseteq V\).\ Then, by definition, \(f(u)=0\) for all \(u\in {^{\perp}W}\), hence \(f\colon V^b\to \R\) passes to the quotient \(\overline{f}\colon V^b/{^{\perp}W}\to \R\).\ 
    Moreover, given that \(W\) is \(\G\)-invariant then also \(^{\perp}W\) is \(\G\)-invariant.\ Indeed if \(u\in {^{\perp}W}\) and \(f\in W\) then for all \(g\in \G\) we have \(f(g\cdot u)=(g^{-1}\cdot f)(u)=0\) since \(g^{-1}\cdot f\in W\) by \(\G\)-invariance.\ Now the \(\G\)-invariance of \(^{\perp}W\) implies that the \(\G\) action on \(V^b\) passes to the quotient.\ This shows that \(V^b/{^{\perp}W}\) is a separable Banach \(\G\)-module.\ We are left to show that the isomorphism \(W\cong (V^b/{^{\perp}W})'\) is \(\G\)-equivariant: for every \({u} \in V^b\) 
    \[g\cdot \overline{f}(u+{^{\perp}W})=\overline{f}(g^{-1}( u+{^{\perp}W}))=\overline{f}(g^{-1} u+{^{\perp}W})= \overline{g\cdot f}(u+{^{\perp}W}).\]
    Similarly, using the \(\G\)-invariance of \(^{\perp}W\), we can show that the map 
    \begin{eqnarray*}
        \phi\colon V/W\to (^{\perp}W)'\\
        f+W\mapsto f_{|^{\perp}W},
    \end{eqnarray*}
    is well defined and \(\G\)-equivariant.\ Indeed, if \(f,g\in V\) and \(f-g\in W\), we have \((f-g)_{|^{\perp}W}=0\) and for every \(g\in \G\) we have \(g\cdot \phi(f+W)=g\cdot f_{|^{\perp}W}=(g\cdot f)_{|^{\perp}W}.\)
\end{proof}

\begin{exa}[the module \(L^{\infty}_{{_W*}}(\Omega,V)\)]\label{Linfcoef}
        Given a standard Borel probability \(\G\)-~space \(\Omega\) and a coefficient \(\G\)-module \(V\),  we denote with \(L_{_W*}^{\infty}(\Omega,V)\) the space of classes of essentially bounded \(V\)-valued weak-* measurable
    maps on \(\Omega\) endowed with the essential supremum norm.\ 
    We endow this space with the \(\G\)-action defined by \[(g\cdot f ) (s) = gf(g^{-1}s) \; \text{for every}\; g \in \G \;\text{and}\; f\in L_{_W*}^{\infty}(\Omega,V).\]
    As proved by Monod, the module \(L_{_W*}^{\infty}(\Omega,V)\) is a coefficient \(\G\)-module, predual to \(L^1(\mu;V^b)\) where \(\mu\) is the probability measure on \(\Omega\) and \(V^b\) is the predual of \(V\) \cite[Corollary 2.3.2]{Mon}.\ 
    Using Lemma \ref{G-mod} we obtain that  the restriction to the \(\G\)-invariants \(L_{_W*}^{\infty}(\Omega,V)^{\G}\), is also a coefficient \(\G\)-module.\\
    We point out that the space \(L_{_W*}^{\infty}(\Omega,V)\) is \emph{not} separable (unless \(\Omega\) is discrete and finite).\ In order to consider a class of coefficient modules that allows us to include \(L_{_W*}^{\infty}(\Omega,V)\) we will define the larger class of semi-separable modules (Section \ref{2}).
\end{exa}

\subsection{Bounded cohomology}
In this section we recall the notions of \emph{bounded cohomology} and \emph{bounded acyclicity} for discrete groups.\ We refer the reader to the books by Monod \cite{Mon} and Frigerio \cite{fri} and the article by Moraschini and Raptis \cite{m.rap} for further details.

\begin{df}[{bounded cohomology}]
    Let \(\G\) be a group and let \(V\) be a Banach \(\G\)-module.\ For all \(n\geq 0\) we define the \emph{cochain complex of bounded functions}:
\begin{eqnarray*}
    C_b^n(\G;V)\coloneqq\{f\colon \G^{n+1}\to V\; |\; ||f||_{\infty}\coloneqq\sup_{g_0,\dots g_n}||f(g_0,\dots,g_n)||_V<+\infty\}
\end{eqnarray*}
endowed with the simplicial coboundary operator
\begin{eqnarray*}
    \delta^n \colon C_b^n(\G; V) \to C_b^{n+1}(\G; V)\end{eqnarray*}
    \begin{eqnarray*}
    \delta^n(f)(g_0, \dots , g_{n+1})\coloneqq\sum_{i=0}^{n+1}
(-1)^if(g_0, \dots, \hat{g_i},\dots , g_{n+1}).
\end{eqnarray*}
   The module \(C_b^n(\G; V)\) is endowed with the following \(\G\)-action: For all \(f\in C_b^n(\G; V)\) and \(\gamma \in \G\), we have
   \[(\gamma\cdot f)(g_0,\dots, g_n)\coloneqq\gamma f(\gamma^{-1}g_0,\dots,\gamma^{-1}g_n).\]
    We define the \(n^{th}\)-\emph{bounded cohomology} of \(\G\) \emph{with coefficients in} \(V\) as the quotient \[H_b^n(\G;V)\coloneqq Z_b^n(\G,V)/B_b^n(\G,V) \text{,    where   }\]
    \[Z_b^n(\G,V):=\ker\delta^n \cap  C_b^n(\G,V)^{\G} \; \; \; \text{and}  \;\;\; B_b^n(\G,V):=\delta^{n-1} ( C_b^{n-1}(\G,V)^{\G}).\]
    The \(\ell^{\infty}\)-norm on \(C_b^n(\G;V)\) descends to a semi-norm in bounded cohomology by taking the infimum of the norm of the representatives in the class, i.e.\ for all \(\alpha \in H_b^n(\G,V)\) we have
    \[ ||\alpha||_{\infty}=\inf\{||f||_{\infty}\;\big\vert\; f\in Z_b^n(\G,V) \; \mbox{and} \; [f]=\alpha \}.\]
   \end{df}

\begin{df}[comparison map]
The inclusion of the complex of bounded functions into the complex of (possibly) unbounded functions \(\iota\colon C_b^*(\G;V)\to C^*(\G;V)\), induces a map in homology
\[c^n\colon H^n_b (\G;V ) \to H^n(\G;V )\]
for all \(n\geq 0\)
called the \emph{comparison map}. This map is
neither injective nor surjective in general \cite[Chapter 1.6]{fri}. We denote the kernel of \(c^n\) by \(EH_b^n(\G;V )\).
\end{df}
In this paper we are interested in groups with vanishing bounded cohomology:

\begin{df}[boundedly acyclic module]
    Let \(\Gamma\) be a discrete group and let \(n \geq 1\) be an integer
or \(n = \infty\).\ A Banach \(\Gamma\)-module \(V\) 
 is \emph{boundedly \(n\)-acyclic} if
\(H^k_b(\Gamma; V )=0\) for every \(1 \leq k \leq n\).\ We say that \(V\) is \emph{boundedly acyclic} if \(V\) is
boundedly \(\infty\)-acyclic.
\end{df}

\begin{df}[\(\mathfrak{X}\)-boundedly acyclic group] Let \(\G\) be a discrete group and \(\mathfrak{X}\) be a subclass of the class of Banach modules and let \(n \geq 1\) be an integer or \(n =\infty\).\ We say that
\(\G\) is \emph{\(\mathfrak{X}\)-boundedly \(n\)-acyclic} if \(H_b^k(\G; V) = 0\) for every \(1 \leq k \leq n\) and every
Banach \(\G\)-module \(V\) in \(\mathfrak{X}\).\ The group \(\G\) is said to be \emph{\(\mathfrak{X}\)-boundedly acyclic}
if it is \(\mathfrak{X}\)-boundedly \(\infty\)-acyclic. 
\end{df}

See Section \ref{5} for a non-exhaustive list of examples of boundedly acyclic groups.

\section{Semi-separable modules}\label{2}

As anticipated, in this section we extend the definition of separable module, in order to define a subclass of coefficient modules that includes the Banach modules~\(L_{_W*}^{\infty}(\Omega,V)\) introduced in Example \ref{Linfcoef} for every standard Borel probability space~\(\Omega\) and every separable coefficient module \(V\).\ The main definition is due to Monod \cite{MonSemi}:

\begin{df}[semi-separable module]
    Let \(\G\) be a group and let \(V\) be a coefficient \(\G\)-module.\ We say that \(V\) is \emph{semi-separable} if there exists a separable coefficient \(\G\)-module \(U\) and an injective dual \(\G\)-morphism \(i\colon V\to U\).\
    We denote by \X{s} the class of all semi-separable coefficient \(\G\)-modules.
\end{df}

The following technical lemma shows that the class of semi-separable coefficient modules is closed with respect to taking the restriction of the action to a subgroup \(H\) of \(\G\) and to taking the submodule of \(H\)-invariants with the action of the quotient \(\G/H\) when \(H\) is normal.

\begin{lem}\label{ssep inv}
    Let us consider the following short exact sequence of discrete groups \[1 \to H\xrightarrow{i} \G\xrightarrow{p} Q\to 1.\] Let \(V\) be a semi-separable coefficient \(\G\)-module.\ Then \(V\) is a semi-separable module with the restricted \(H\)-action and \(V^H\) is semi-separable with the induced \(Q\)-action.
\end{lem}
\begin{proof}

    Let \(U\) be a separable coefficient \(\G\)-module and let \(\iota\colon V\to U\) be an injective dual \(\G\)-morphism witnessing the definition of semi-separability.\ 
    Then \(\iota\colon V\to U\) is also an injective dual morphism of \(H\)-modules if we restrict the action by \(\G\) to \(H\).\ This shows that \(V\) is a semi-separable Banach \(H\)-module.\\
    Let us show now that \(V^H\) is a semi-separable Banch \(Q\)-module. First notice that \(U^H\) is a weak*-closed subspace of \(U\) by Lemma \ref{closed}.\
    Recall that a metric space is separable if and only if it is second countable, and hence separability passes to subspaces.\ This shows that the submodule \(U^H\) is separable since \(U\) is a separable metric space.\  
    We want to show that the image of \(V^H\) under the injection \(\iota\) is contained in \(U^H\).\ This is true as for every
    \(v\in V^H\) and \(h\in H\) we have \[h\cdot \iota(v)=i(h)\cdot \iota(v)=\iota(i(h)\cdot v)=\iota(v).\]
    Recall that \(V^H\) and \(U^H\) can be considered as \(Q\)-modules as in Definition 
    \ref{quotmod}.\ We are going to show now that \(\iota_{|V^H}\colon V^H\to U^H\) is a \(Q\)-morphism.\ Let \(g H\) be a coset in \(Q\cong \G/H\) and \(v\in V^H\).\ Then, by definition of the action by \(Q\) and using the fact that \(\iota\) is a \(\G\)-morphism, we get \[gH\cdot \iota(v)=g\cdot \iota (v)=\iota(g\cdot v)=\iota(gH\cdot v).\]
    We are left to show that \(\iota_{|{V^H}}\) is a dual \(\G\)-morphism.\
    By Remark \ref{dual}, in order to prove the duality of the \(Q\)-morphism \(\iota_{|V^H}\colon V^H\to U^H\) it is sufficient to show that it is weak-* continuous.\ This is readily verified as \(\iota\) is dual and so weak-* continuous, whence the restriction \(\iota_{|V^H}\colon V^H\to U\) is still weak-* continuous. Moreover since \(U^H\) is weak-* closed, for every weak-* closed subset \(W\subset U^H\), \(W\) is also weak-* closed in \(U\) and so it has closed preimage in \(V^H\) by weak-* continuity of \(\iota_{|V^H}\).\
    This proves that \(\iota_{|V^H}\colon V^H\to U^H\) is a dual, injective, \(Q\)-morphism to a separable coefficient \(Q\)-module which concludes the proof.
\end{proof}

\begin{quest}\label{quest}
    By definition if a group is \X{s}-boundedly acyclic then it is also \X{}-boundedly acyclic.\ It is an open question whether the converse is also true.
\end{quest}
The class of \X{s}-boundedly acyclic groups is strictly larger than the class of amenable groups: just take \(\G\wr \Z\) with \(\G\) non-amenable, cf.\ Theorem \ref{A}.\

\section{Group actions}\label{3}

In this section we collect definitions and results on Zimmer-amenable actions and ergodic actions on standard Borel porobability spaces.\ In particular we will see that one can use \emph{non-singular} (i.e.\ measure class preserving) highly ergodic and Zimmer-amenable actions to compute bounded cohomology of groups.

\subsection{Zimmer-amenable actions}
\begin{df}[Zimmer amenable space]
    Let \(\G\) be a group.\ Let \(\Omega\) be a
standard Borel probability space, equipped with a non-singular \(\G\)-action.\
A \emph{conditional expectation}
\[m\colon L^{\infty}(\G \times \Omega, \R) \to L^{\infty}(\Omega, \R)\]
is a norm one linear map such that:
\begin{enumerate}
    \item \(m(\mathbbm{1}_{\G\times \Omega})= \mathbbm{1}_\Omega\);
    \item For all \(f \in L^{\infty}(\G\times \Omega, \R)\) and every measurable subset \(A \subset \Omega\), it
holds \(m(f \cdot \mathbbm{1}_{\G\times A}) = m(f) \cdot \mathbbm{1}_A\).
\end{enumerate}
We say that \(\Omega\) is a \emph{Zimmer-amenable} \(\G\)-space, if there exists a conditional
expectation \(m\colon L^{\infty}(\G\times \Omega,\R) \to L^{\infty}(\Omega,\R)\) that is moreover \(\G\)-equivariant (on
\(\G\times \Omega\) we consider the diagonal \(\G\)-action).
\end{df}

\begin{exa}[amenable groups]
    A discrete group \(\G\) is amenable if the single point \(\{\ast\}\) is a Zimmer-amenable \(\G\)-space.
\end{exa}

Burger and Monod proved a characterization of Zimmer-amenable actions that is based on the notion of relatively injective modules:

\begin{df}[relatively injective Banach module]
A Banach \(\G\)-module \(V\) is \emph{relatively injective} (with
respect to \(\G\)) if for every \(\G\)-morphism \(\iota \colon A \to B\) of Banach \(\G\)-modules \(A\) and \(B\), that admits a (non necessarily \(\G\)-equivariant) linear retraction
\(\sigma \colon B \to A\) with \(||\sigma||\leq 1\) and \(\sigma\circ\iota = Id_A\),
and for every \(\G\)-morphism \(\alpha\colon A \to V\),
there is a \(\G\)-morphism \(\beta\colon B \to V\) satisfying \(\beta \circ \iota=\alpha\) and \(||\beta||\leq||\alpha||\).
\end{df}

\begin{teo}[{\cite[Theorem 1]{BurMon}}]\label{relinj}
    Let \(\G\) be a countable group and let \(\Omega\) be a
standard Borel probability \(\G\)-space.\ The following assertions are equivalent\begin{enumerate}
    \item The \(\G\)-action on \(\Omega\) is Zimmer-amenable;
    \item The Banach \(\G\)-module \(L^{\infty}_{_{W^*}}(\Omega^{n+1}, V)\) is relatively injective for all
\(n \geq 0\) and every coefficient \(\G\)-module \(V\).
\end{enumerate}
\end{teo}

Relatively injective modules play an important role in the theory of bounded cohomology, for example they can be used to compute bounded cohomology via strong resolutions. In particular if \(\G\) is a discrete group
and \(V\) a relatively injective coefficient \(\G\)-module, then it holds that \(H^n_b(\G, V)=0\) for
all \(n\geq 1\) {\cite[Corollary 1.5.5]{BurMon}}. This result, together with the characterization of Zimmer-amenable actions (Theorem \ref{relinj}), gives the following corollary:

\begin{cor}\label{relinj0}
    Let \(H\) be a discrete group and let \(\Omega\) be a standard probability space with a Zimmer-amenable \(H\)-action.\ Let \(V\) be a coefficient \(H\)-module.\
    Then \[H_b^n(H,L^{\infty}_{_{W^*}}(\Omega^{k+1},V))=0\] for all \(n\geq 1\) and \(k\geq 0.\)
\end{cor}

\subsection{Ergodic actions} 
\begin{df}[ergodicity with coefficients]
   Let \(\G\) be a discrete group and let \(\Omega\) be a standard Borel
probability space, equipped with a non-singular \(\G\)-action.\ We say that
the action is \emph{ergodic} if every \(\G\)-invariant subset of \(\Omega\) has either zero or full
measure.\ Equivalently,  the \(\G\)-action on \(\Omega\) is ergodic if and only if every essentially
bounded measurable \(\G\)-invariant function \(f\colon \Omega \to \R\) is essentially constant \cite[Proposition 2.1.11]{zimmerergodic}.\\
Let \(V\) be a Banach \(\G\)-module.\ 
We say that the action of \(\G\) on
\(\Omega\) is \emph{ergodic with coefficients} in \(V\) if every essentially bounded, weak*-measurable \(\G\)-equivariant function \(f\colon \Omega \to V \) is essentially constant, i.e.\  \(L^{\infty}_{_W*}(\Omega, V)^{\G}\cong V^{\G}\).
\end{df}

\begin{df}[high ergodicity] Let \(\G\) be a group and let \(\Omega\) be a standard
Borel probability space with a non-singular \(\G\)-action.\ We say that the action
of \(\G\) on \(\Omega\) is \emph{highly ergodic} if the diagonal action \(\G\curvearrowright \Omega^n\) is ergodic for all
integers \(n \geq 1\).
\end{df}

It is a classical fact that double ergodicity implies that all \(\G\)-invariant bounded functions to separable modules are essentially constant.\ Monod showed that this is true also for semi-separable modules \cite[Lemma 5.5]{MonSemi}:

\begin{prop}\label{erg semi-sep}
    Let \(\Omega\) be a standard probability space, and let \(\G\curvearrowright \Omega\) be a measure-class preserving action.\ Let \(V\) be a coefficient \(\G\)-module and let \(m\geq 1\) be an integer.\
    If the diagonal action \(\G\curvearrowright \Omega^{2m}\) is ergodic and \(V\in \) \X{s}, then \(L^{\infty}_{_W*}(\Omega^m;V)^{\G}\cong V^{\G}\).
\end{prop}

\begin{rk}[ergodicity and subgroups]\label{subergo}
    Let \(\Gamma\) be a group and let \(\Omega\) be
a standard Borel probability \(\G\)-space.\ Suppose that there exists a subgroup
\(H\leq \G\) such that the restricted action \(H \curvearrowright \Omega\) is ergodic.\ Then, the action
\(\G \curvearrowright \Omega\) is also ergodic.\ Indeed, if all the \(H\)-invariant functions are essentially
constant also the \(\G\)-invariant functions must be so.
\end{rk}

The following example by Monod provides us with a large class of semi-separable coefficient modules.

\begin{lem}\label{linf ssep}
    Let \(\Omega\) be a standard probability space, let \(H\) be a discrete group with a measure preserving action on \(\Omega\) and let \(V\) be a semi-separable coefficient \(H\)-module.\ Then \(L^{\infty}_{_{W^*}}(\Omega;V)\) is a semi-separable coefficient \(H\)-module with the action given by \[g\cdot f(w)=gf(g^{-1}w)\] for every \(g\in H\), for every \(w \in \Omega.\)
\end{lem}

 \begin{proof}      
    Let be \(V\) is semi-separable coefficient \(H\)-module.\ This is a special case of of a result by Monod
    {\cite[Lemma 3.13]{MonSemi}} which states the semi-separability of \(L_{_W*}^{\infty}(\Omega,V)\) when seen as a coefficient \(H\)-module with the action given by \[g\cdot f(\omega)\coloneqq\beta(g^{-1},\omega)f(g^{-1}\omega)\] where \(g\in H\), \(f\in L_{_W*}^{\infty}(\Omega,V)\) and  \(\beta:H\times \Omega \to H\) is a measurable cocycle.
    In our case we consider the cocycle \begin{eqnarray*}
       \beta\colon H\times \Omega &\to& H\\
        (g,w)&\mapsto& g.
 \end{eqnarray*} Indeed \(\beta\) is measurable as \(\beta^{-1}(g)=g\times \Omega\) for all \(g\in H\) and \[\beta(g_1g_2,w)=g_1g_2=\beta(g_1,g_2\cdot w)\beta(g_2,w)\] for all \(g_1,g_2\in H\) and \(w\in \Omega\).\
 Moreover, the \(H\)-action on \(V\) obtained with this cocycle coincides with the usual action of \(H\) on \(L^{\infty}_{_{W^*}}(\Omega, V)\) as \[\beta(g^{-1},w)^{-1}f(g^{-1}\cdot w)=g \cdot f(g^{-1}\cdot w)\] for all \(g\in H\) and \(f\in L^{\infty}_{_{W^*}}(\Omega, V).\)
 \end{proof}

\section{Bounded acyclicity of permutational wreath product}\label{4}

The goal of this section is to prove Theorem~\ref{A} of the introduction, which states the \X{s}-bounded \(n\)-acyclicity of the permutational wreath product \(\G\wr_XA\) where \(A\) is \X{s}-boundedly \(n\)-acyclic and acts with infinite orbits on \(X\).\ This result extends a previous theorem by Monod that proved the \X{}-bounded acyclicity of wreath products of the form \(\G\wr \Z\) for any group \(\G\) \cite[Theorem 3]{lamp}.\ The extension of Monod's result to the case of permutational wreath product \(\G\wr_X A\) with \(A\) amenable was already mentioned by Monod and proved in detail by Campagnolo, Fournier-Facio, Lodha, and Moraschini \cite[Corollary 3.4]{criterion}.\\
The most relevant difference here is that we are able to remove the hypothesis of amenability and substitute it with the much more general notion of \X{s}-bounded acyclicity.\medskip

\begin{df}[Permutational wreath product]
    Let \(\Gamma\) and \(A\) be groups and let \(X\) be an \(A\)-set.\ The \emph{permutational
wreath product} \(\Gamma\wr_X A\), of \(\G\) with \(A\), is defined as the semi-direct product \((\bigoplus_{x\in X}\Gamma)\rtimes_{\alpha} A,\) where the action \(\alpha\) of \(A\) on the
group \(\bigoplus_{x\in X} \Gamma\) is by permutation of the factors.\ The case \(X=A\) with the
left multiplication action recovers the regular restricted wreath product, or simply \emph{wreath
product} \(\Gamma \wr A\).
\end{df}

Following the approach of Campagnolo, Fournier-Facio, Lodha, and Moraschini, we introduce a similar setup for working with permutational wreath products.

\begin{set}
    [Permutational wreath products with acyclic acting group]\label{setup}
We consider the following setup:
\begin{itemize}
    \item Let \(\Gamma\) and \(A\) be two countable groups, such that \(A\) is \X{s}-boundedly \(n\)-acyclic;
    \item Let \(X\) be a countable \(A\)-set such that all the \(A\)-orbits are infinite;
    \item Let \((Y_0, \mu_0)\) be the standard Borel probability \(\Gamma\)-space obtained by
    endowing \(\Gamma\) with a distribution of full support;
    \item Let \(Y \coloneqq Y_0^X\)
     be the standard Borel probability space endowed with
    the product measure;
    \item The permutational wreath product \(\Gamma \wr_X A\) acts on \(Y\) as follows:
    \(\bigoplus_{x\in X} \Gamma\) acts coordinate-wise and \(A\) acts by shifting the coordinates.
\end{itemize}

\end{set}

\begin{rk}\label{Hssep}
    Set \(H:=\bigoplus_{x\in X}\G\) and \(G\coloneqq \G\wr_X A\).\ In this setup the space \(L^{\infty}_{_W*}(Y;V)^H\) is a semi-separable coefficient \(A\)-module for every \(V\) semi-separable coefficient \(G\)-module.\ 
    Indeed first note that we can identify \(A\) with a subgroup of \(G\), so every coefficient Banach \(G\)-module is also a coefficient Banach \(A\)-module by restriction (Definition
    \ref{quotmod}).\
    Moreover, the action of \(A\) on \(Y\) is measure preserving as it acts by shifting the coordinates.\ By Lemma \ref{linf ssep} the space \(L^{\infty}_{_W*}(Y;V)\) is a semi-separable coefficient \(A\)-module.\ So we are left to show that the inclusion \(L^{\infty}_{_W*}(Y;V)^H\hookrightarrow L^{\infty}_{_W*}(Y;V)\) is dual and \(A\)-equivariant knowing that is dual and \(G\)-equivariant.\ 
    But the \(G\)-action on \(L^{\infty}_{_W*}(Y;V)\) descends to the quotient action of \(A=G/H\) on \(L^{\infty}_{_W*}(Y;V)^H\).\ Thus the inclusion is \(A\)-equivariant.\ 
    
\end{rk}
   
\begin{prop}[Highly ergodic]\label{erg}
    In the situation of Setup \ref{setup}, the action of \(\Gamma\wr_X A\) on \(Y\) is highly ergodic.
\end{prop}
\begin{proof}
    This follows from the high ergodicity of the generalized Bernoulli shift for actions with infinite orbits \cite[Proposition 2.15]{criterion} together with the high ergodicity of subgroups (Remark \ref{subergo}).
\end{proof}

\begin{prop}[{\cite[Proposition 2.10]{criterion}}]\label{zam}
    In the situation of Setup \ref{setup}.\ The action of \(\bigoplus_{x\in X}\Gamma\) on \(Y\) is Zimmer-amenable.
\end{prop}
\medskip
In order to prove the upcoming theorem, we will need the following result about spectral sequences in bounded cohomology:

\begin{teo}[Hochschild-Serre spectral sequence {\cite[Proposition 12.2.2]{Mon}}] \label{spec}
    Let \[1\to H \to G \to A \to 1\] be a short exact sequence of discrete groups and let \(V\) be a Banach \(G\)-module.\ Then there exist a first-quadrant spectral sequence \((E_*; d_*)\), called the \emph{Hochschild-Serre
spectral sequence}, converging to the bounded cohomology of \(G\) with coefficients in~\(V\).\\
Moreover, if for every \(q \in \N\) the semi-norm on \(H_b^q(H; V )\) is a
norm, then for each \(p \in \N\) the identification
\[E_2^{pq}\cong H_b^p(A;H_b^q(H;V))\] holds, where \(H_b^q(H;V)\) is a Banach \(A\) module with the conjugation action.
\end{teo}

\begin{rk}\label{azioni}
    Recall that the conjugation action by \(A=G/H\) on \(H_b^q(H;V)\) is given by \(gHf(h)=gf(g^{-1}hg)\) for \(gH\in A\) and \(h\in H\).\ In degree zero if \(f\colon~H~\to~V~\in~Ker(\delta^0)\) then \(f\) is constant and \(H\)-invariant.\ Hence \(gH\cdot~ f(h)~=~gf(h)\) for every \(gH\in A\) and every \(h\in H\).\ In particular the \(A\)-action on \(H^0_b(H;V)=V^H\) coincides with the action induced by the \(G\)-action on \(V^H\) to the quotient \(G/H=A\).
\end{rk}

\medskip
We are now ready to prove the main result of this section:

\begin{teo}[Theorem \ref{A}]\label{acyclicwr}
    In the situation of Setup \ref{setup}, the wreath product \(\Gamma\wr_X A\) is \X{s}-boundedly \(n\)-acyclic.
\end{teo}
\begin{proof}
    Set \(G:=\Gamma\wr_X A\).\ We need to show that for every semi-separable coefficient \(G\)-module \(E\) and for every \(n\geq 1\) we have \(H_b^n(G;E)=0\). \medskip \\
    \textbf{Claim:} The \(G\)-module \(L^{\infty}_{_W*}(Y^{k},E)\) is Banach and  boundedly \(n\)-acyclic for every \(k\geq 1\).\\
    \textit{Proof of Claim.} From Example \ref{Linfcoef} we know that \(L^{\infty}_{_W*}(Y^{k},E)\) and \(L^{\infty}_{_W*}(Y^{k},E)^G\) are Banach \(G\)-modules for every \(k\geq 1\).\ We have to show that the bounded cohomology \(H_b^n(G;L^{\infty}_{_W*}(Y^{k},E))\) vanishes for all \(n\geq1\) and \(k\geq 1\).\ 
    To this end we use the Hochschild-Serre spectral sequence described in Theorem \ref{spec}, with \(V=L^{\infty}_{_W*}(Y^{k},E)\) for any \(k\in \N_{>0}\) and \(H=\bigoplus_{x\in X} \Gamma\).\
    First we check that \(H_b^q(H;V)\) is Banach, where \(V\) is endowed with the restriction \(H\)-action.\ This is always true.\ Indeed, since the action of \(H\) on \(Y\) is Zimmer-amenable (Proposition \ref{zam}), Theorem \ref{relinj} implies that the \(H\)-module \(L^{\infty}_{_W*}(Y^{k},E)\) is relatively injective for all \(k\geq 1\).\ Hence by Corollary \ref{relinj0} we have for every \(k\geq 1\) \[H_b^q(H;L^{\infty}_{_W*}(Y^{k},E))=\begin{cases}
        0, \quad \text{ if } q> 0\\
        L^{\infty}_{_W*}(Y^{k},E)^{H}, \quad \text{ if } q=0.
    \end{cases}\] 
    Hence \(H_b^q(H;V)\) is Banach for every \(q\geq 0\).\ This shows that we can identify the second page of the 
    Hochschild-Serre spectral sequence as follows (Theorem \ref{spec}):
    \[E_2^{pq}=H_b^p(A;H_b^q(H;V)).\]
    By Remark \ref{Hssep}  have that each \(L^{\infty}_{_W*}(Y^{k},E)^{H}\) is a Banach semi-separable coefficient \(A\)-module for all \(k\geq1\) (Remark \ref{azioni}).\ 
  Now by hypothesis \(A\) is \X{s}-boundedly \(n\)-acyclic and hence 
  \begin{eqnarray*}
    E_2^{pq}&=&H_b^p(A;H_b^q(H;L^{\infty}_{_W*}(Y^{k},E)))\\&=&\begin{cases}
      H_b^p(A;L^{\infty}_{_W*}(Y^k;E)^H), \; \text{ if } q=p=0 \text{ or if } q=0 \text{ and } p>n \\
      0,  \; \text{ if } q>0 \text{ or if } 0< p\leq n \text{ and } q= 0   
  \end{cases}
  \end{eqnarray*}
   Since the spectral sequence has only one non-zero row, the convergence of the spectral sequence results in the isomorphism \begin{eqnarray*}
  E^{p0}=H_b^p(A;H_b^0(H;L^{\infty}_{_W*}(Y^{k},E)))\cong H_b^p(G;L^{\infty}_{_W*}(Y^{k},E))
  \end{eqnarray*} for all \(p\geq1 \) and \(k\geq 1\).\ In particular \[H_b^p(G;L^{\infty}_{_W*}(Y^{k},E))\cong 0\] for \(0<p\leq n.\)
  This concludes the proof of the bounded \(n\)-acyclicity of \(L^{\infty}_{_W*}(Y^{k},E)\) as \(G\)-module for every \(k\in \N_{>0}\), whence we proved the claim.\
    \medskip\\ 
    We can now prove that for every semi-separable coefficient \(G\)-module \(E\), we have \(H_b^p(G;E)=0\) for all \(0<p\leq n.\) \\
    Given a resolution of a coefficient \(G\)-module \(E\) \[0\to E\to I^0\to I^1\to I^2\to \dots\] 
    such that \(I^0,\dots, I^n\) are boundedly \(n\)-acyclic \(G\)-modules, it is known that the cochain complex of \(G\)-invariants
    \[0\to (I^0)^G\to (I^1)^G\to \dots\] 
    computes the bounded cohomology of the group \(G\) with \(E\) coefficients in degrees \(0\leq p\leq n\): 
    \(H^p_b(G;E)\cong H^p(I^{\bullet,G})\) \cite[Proposition 2.5.4]{m.rap}.\
    Consider the exact sequence 
    \[0\to E\to L^{\infty}_{_W*}(Y,E) \to L^{\infty}_{_W*}(Y^2,E)\to \dots.\]
    By Claim, this resolution consists of boundedly \(n\)-acyclic modules and so we have the isomorphism \[H_b^p(G;E)\cong H^p(L^{\infty}_{_W*}(Y^{\bullet};E)^G) \mbox{  for  }  0\leq p\leq n.\]
    Moreover, since the action of \(G\) on \(Y\) is highly ergodic, by Proposition \ref{erg semi-sep} we have \[L^{\infty}_{_W*}(Y^{k};E)^G\cong E^G \mbox{ for all}\; k\geq 1.\] Hence the bounded cohomology of \(G\) with \(E\)-coefficients can be computed in degrees \(0\leq p\leq n\) via the resolution \(0\to E^G\to E^G\to \dots\) where the boundary operators are the identity in odd degrees and identically \(0\) in even degrees.\ This proves that \(H_b^p(G;E)=0\) for all \(1\leq p\leq n\), whence \(G\) is \X{s}-boundedly \(n\)-acyclic.  
\end{proof}

\begin{rk}[less coefficients]\label{sololinf}
By following the proof of Theorem \ref{acyclicwr}, one sees that, given a semi-separable coefficient (\(\G\wr_X A\))-module \(E\), in order to have the vanishing \(H_b^k(\G\wr_XA;E)=0\) for every \(1\leq k\leq n\), it is sufficient to suppose the group \(A\) to be boundedly \(n\)-acyclic only with respect to the coefficient \(L_{_W*}^{\infty}(Y^{k};E)^H\) for every \(k\in \N_{>0}\), where \(H\coloneqq ~\bigoplus_{x\in X}\G\).
\end{rk}

\begin{cor}[modules with no invariants]\label{invariants}
    Let \(A\) be a countable group acting on a countable set \(X\) with infinite orbits and let \(\G\) be any group. Suppose that \(H_b^k(A;V)=0\) for all \(1\leq k\leq n\) and for every semi-separable coefficient \(\G\wr_X A \)-module with \(V^A=0\).\ Then,  \(H_b^k(\G\wr_XA;V)=0\) for all \(1\leq k\leq n\).
\end{cor}
\begin{proof}
    By Remark \ref{sololinf} it is sufficient to show that \(H_b^k(A;L_{_W*}^{\infty}(Y^{i};V)^H)=0\), for all \(i\geq 1\) and for all \(1\leq k\leq n\), where \(H\coloneqq ~\bigoplus_{x\in X}\G\), and \(Y\) as in Setup \ref{setup}.\ By ergodicity of the action of \(\G\wr_XA\) on \(Y^i\) (Proposition \ref{erg}), we have \[(L_{_W*}^{\infty}(Y^{i};V)^H)^A=L_{_W*}^{\infty}(Y^{i};V)^{\G\wr_XA}=V^{\G\wr_XA}=(V^H)^A=0.\] This shows that \(L_{_W*}^{\infty}(Y^{i};V)^H\) has no non-trivial \(A\)-invariants, and so by assumption  \(H_b^k(A;L_{_W*}^{\infty}(Y^{i};V)^H)=0\). 
\end{proof}

\begin{rk}[uncountable case]
    As stated in the introduction, Theorem \ref{acyclicwr} can be extended to the case of permutational wreath products \(\G\wr_XA\) where \(\G\) is uncountable via standard techniques \cite[Proposition 12]{lamp}.
\end{rk}

\section{Examples}\label{5}
We present large classes of new examples arising from our main theorem.\ We begin with a straightforward example obtained iterating Theorem \ref{A} (Example \ref{iteratedex}) and continue with more sophisticated examples such as products of groups acting without invariants (Example \ref{noinvariantsex}) where we use Remark \ref{sololinf}, and wreath products with lattices (Example \ref{lattices}), where we use Corollary \ref{invariants}.

\begin{exa}[iteration]\label{iteratedex} The conditions of Theorem \ref{A} allow for iteration.\\ Indeed let \(A\) be a \X{s}-boundedly \(k\)-acyclic group for some \(k\in \N\).\ Let \(X\) be an \(A\)-set with infinite \(A\)-orbits and let \(\{\G_i\}_{1\leq i\leq n}\) be a family of groups.\  
        Then the wreath product \[\G_n\wr_X(\G_{n-1}\wr_X(\cdots\wr_X (\G_2\wr_X(\G_1\wr_XA))\cdots))\]
        is still \X{s}-boundedly \(k\)-acyclic.\ The action of the wreath products on \(X\) is the action of \(A\), that is, for \((g_i,\dots(g_1,a))\in \G_i\wr_X\cdots( \G_1\wr_X~A)\) and \(x\in X\) we have \((g_i,\dots(g_1,a))x=ax\).\ So it has infinite orbits.\ 
\end{exa}

\begin{exa}[wreath product with Thompson F]\label{Thompson F}
     Let us consider Thompson group \(F\) and the set \(X\coloneqq(0,1)\cap \Z[1/2].\) The group \(F\) is \X{s}-boundedly acyclic as indicated by Monod \cite[Section 4.3]{lamp}, and \(F\) acts diagonally on \(X^i\) with finitely many orbits for every \(i\in \N\).\ 
     We can use Theorem \ref{A} to obtain that, for every group \(\G\), the wreath product \(\G\wr_XF\) is \X{s}-boundedly acyclic. 
\end{exa}

\begin{exa}[wreath product with groups of homeomorphisms of the line]
   Let \(H\) be a boundedly supported group of orientation-preserving
    homeomorphisms of the line and assume that \(H\) is proximal (i.e. for all reals \(a < b\) and \(c < d\) there exists \(h \in H\) such that \(h \cdot a < c < d < h \cdot b\) ).\ Then, as proved by Fournier-Facio and Rangarajan \cite[Corollary 5.2]{Ulamff}, there exists a subgroup \(H_0\leq H\) for which 
    \begin{enumerate}
        \item there exists an element \(h \in H\) such that the groups \({h^i
    H_0h^{-i}\colon i \in \Z}\) pairwise commute;
    \item every finite subset of \(H\) is contained in some conjugate of \(H_0\). 
    \end{enumerate} 
    These conditions imply that the group H is boundedly acyclic~\cite[Corollary 5 and Section 4.3]{lamp}.\ Using Theorem \ref{A} we obtain that, for every group \(\G\) and every countable \(H\)-set \(X\) where \(H\) acts with infinite orbits, the wreath product \(\G\wr_XH\) is \X{s}-boundedly acyclic.
\end{exa}

The following remark allows to check the \X{s}-bounded acyclicity of a group just by looking at its bounded cohomology with real coefficients and with invariant coefficients.

\begin{rk}[coefficients with no-invariants or trivial]\label{noinvR}
Let \(H\) be a group and let \(V\) be a coefficient \(H\)-module, the short exact sequence \[0\to V^{H}\to V\to V/V^{H}\to 0\] induces a long exact sequence in bounded cohomology \cite[Theorem 2.31]{m.rap}:
\[0\to H_b^0(H;V^{H})\to H_b^0(H;V) \to H_b^0(H;V/V^{H})\to\cdots\]
Notice that \(V^{H}\) is a closed sub-module and hence the quotient is still a coefficient module (Lemma \ref{closed}).\
Using this long exact sequence we get that the bounded cohomology of \(H\) with coefficients in \(V\) vanishes if it vanishes for \(V^{H}\) and \(V/V^{H}\) coefficients.\ Now using a result by Moraschini and Raptis we have that the vanishing of bounded cohomology with trivial real coefficients implies the vanishing of bounded cohomology for any trivial coefficient module \cite[{Proposition 2}]{addendum}.\ Hence, any \(n\)-boundedly \(\R\)-acyclic group is also \(n\)-boundedly \(V^{H}\)-acyclic.\ Moreover, as shown by Monod, the quotient \(V/V^{H}\) has no \(H\)-invariants \cite[Lemma 1.2.10]{Mon}.\ Hence, a boundedly \(\R\)-acyclic group has trivial bounded cohomology with coefficients in \(V\) if and only if it has trivial bounded cohomology with coefficients in \(V/V^H\).
\end{rk}
From this observation we can obtain new interesting examples of \(n\)-boundedly \X{s}-acyclic groups.

\begin{exa}[wreath product with lattices]\label{lattices}
     Let \(\{G_{i}\}_{i\in I}\) be connected, simply connected, almost \(k_{i}\)-simple algebraic groups, where \(I\) is a finite non-empty set and \(\{k_{i}\}_{i\in I}\) are local fields.\ Set \(n=\min\{\mbox{rank}(G_{i})\vert i\in I\}\) and let \(H<\prod_{i\in I} G_{i}\) be a lattice.\ 
     If \(V\) is a semi-separable coefficient \(H\)-module without \(H\)-invariants, then using a result by Monod \cite[Corollary 1.8]{MonSemi} we have \[H_b^j(H;V)=0 \mbox{ for every } 1\leq j<2n.\]
     By Corollary \ref{invariants}, if \(H\) acts on a countable set \(X\) with infinite orbits, and \(\G\) is any group then for every coefficient \(\G\wr_XH\)-module \(V\) such that \(V^H=0\), we have \begin{equation}\label{eq1}
         [H_b^j(\G\wr_XH;V)=0 \mbox{ for every } 1\leq j<2n.\end{equation}
     
     In degree two something more can be said for semi-separable coefficients.\ 
    For example, if \(k_i\) is non-Archimedean, Monod showed that the continuous bounded cohomology of any algebraic group \(G_{i}\) over \(k_i\) with trivial real coefficients vanishes in every positive degree \cite[Theorem A]{monodflat}.\ 
    Using a result by Monod and Shalom \cite[Corollary 2.11]{MonShal} we can compute the second bounded cohomology of an irreducible lattice \(H<\prod_{i\in I}G_i\) from the continuous second bounded cohomology of the factors \(G_{i}\) when each \(G_i\) is connected and locally compact.\ More precisely we have 
     \begin{equation}\label{eq2}  
     H_b^2(H;\R)\cong \bigoplus_{i\in I}H_{bc}^2(G_{i};\R)=0.\end{equation}
     Now, when \(H<\prod_{i\in I}G_i\) with \(G_i\) connected, simply connected, almost \(k_i\)-simple algebraic group with \(n\geq2\) and \(k_i\) non-Archimedean local fields, we can combine the result for real coefficients and for semi-separable coefficient without \(H\)-invariants.\
     Indeed from Equation \ref{eq1} we have that, if \(n\geq2 \) and \(V\) is a semi-separable coefficient \(H\)-module,
     \[H_b^2(H;V/V^{H})=0.\] Combining this with Equation \ref{eq2}, Remark \ref{noinvR} shows that \(H_b^2(H;V)=0\) for every semi separable coefficient module \(V\).\\
     Theorem \ref{A} in this situation proves the following: for every group \(\G\) and every \(H\)-set \(X\) with infinite \(H\)-orbits, the wreath product \(\G\wr_XH\) is \(V\)-boundedly 2-acyclic.
\end{exa}

Using Remark \ref{sololinf} we show a construction of wreath products \(\G\wr_XA\) that have vanishing bounded cohomology in low degree for some coefficients \(V\), even if \(A\) is not required to be \X{s}-boundedly \(n\)-acyclic (nor \(V\)-boundedly \(n\)-acyclic).

\begin{exa}[direct product]\label{noinvariantsex}
     Let \(\{A_i\}_{1\leq i\leq n} \) be a family of countable groups and let \(A\coloneqq A_1\times \dots \times A_n\).\ 
        Let \(\{X_i\}_{1\leq i\leq n}\) be a family of countable sets such that each \(A_i\) acts on \(X_i\) with infinite orbits.\ The group \(A\) acts component-wise on the product \(X\coloneqq X_1\times \dots \times X_n\) and the factors \(A_i\) act on \(X\) by restriction of the action of \(A\).\
        Let \(\G\) be a countable group endowed with a probability measure of full support and let \(V\) be a semi-separable coefficient \(\G\wr_X A\)-module.\ 
         Set \(H\coloneqq \bigoplus_{x\in X} \G\) with the component-wise action on \(\G^X\).\
        The module \(V^H\) is semi-separable with the induced \(A\)-action (Lemma \ref{ssep inv}).\      
        The restriction of the action of \(A\) on \(V^H\) to each factor \(A_i\curvearrowright V^H\) turns \(V^H\) into a semi-separable coefficient \(A_i\)-module for every \(1\leq i\leq n\).\ Assume that the \(A_i\)-invariants of \(V^H\) consist only of the neutral element, \((V^H)^{A_i}=\{0\}\) for all \(1\leq i\leq n\).
         Since for each action \(A\curvearrowright X\) and \(A_i\curvearrowright X\) the orbits are infinite, the generalized Bernoulli shift actions 
         \(A\curvearrowright \G^X\) and \(A_i\curvearrowright \G^X\) are highly ergodic for every \(1\leq i\leq n\)  \cite[Proposition 2.15]{criterion} and hence also the actions of the wreath products \(\G\wr_XA_i\) on \(\G^X\) are highly ergodic (the action is defined as in Setup \ref{setup}).\ By ergodicity of the action for each \(A_i\), we obtain: 
         \begin{eqnarray*}
             (L^{\infty}_{_W*}((\G^X)^k;V)^H)^{A_i}=L^{\infty}_{_W*}((\G^X)^k;V)^{\G\wr_XA_i}\\=V^{\G\wr_XA_i}=(V^H)^{A_i}=\{0\}.
         \end{eqnarray*}
         
        The triviality of the invariants for each term of the product in \(A\) allows to apply a theorem by Monod stating that \[H_b^k(A;L^{\infty}_{_W*}(\G^X;V)^H)=0\]
        for every \(0< k< 2n\) \cite[Theorem 1.9]{MonSemi}.\\
        Following the proof of Theorem \ref{acyclicwr} as described in Remark
        \ref{sololinf} we obtain the \(V\)-bounded \((2n-1)\)-acyclicity for the wreath product:
        \[H_b^k(\G\wr_XA;V)=0 \mbox{ for all } 0<k<2n.\]
\end{exa}

\section{Verbal products}\label{6}

Moran defined the \emph{verbal products}: a new class of group operations that behave similarly to direct sums and free products and extend Golovin's definition of nilpotent products \cite{mor}.\ This notion provides new ways of constructing groups.\ It is of interest to study whether the properties of the free product and the direct sum are carried over to these more
general operations.\ More precisely we look at \emph{verbal wreath product} and bounded acyclicity. In this section we prove Corollary \ref{B}, which provides classes of verbal wreath products for which the bounded cohomology vanishes in all positive degrees and all semi-separable coefficients.

\begin{df}[verbal subgroup]
    Let \(\mathbb{F}_{\infty}\) be the free group in countable generators \(\{x_i\}_{i\in \N}\), we call \emph{words} the elements \(w\in \mathbb{F}_{\infty}\).\ A word is in \emph{n-letters} if it requires at most \(n\) distinct \(x_i\)'s to be written in a reduced form.\\ Let \(\G\) be a group and let \(w(x_{i_1},\dots, x_{i_n})\) be a word in \(n\)-letters, the \emph{evaluation} of \(w\) at \((g_1,\dots,g_n)\in \G^n\) is the element \(w(g_1,\dots,g_n)\in \G.\)\\ 
    For a non-empty subset \(W\) of words of \(\mathbb{F}_{\infty}\), the associated \emph{verbal subgroup} \(W(\G)\) is the subgroup of \(\G\) generated by the evaluation of all the elements of \(W\) by the
elements of \(\G\).
\end{df}

\begin{df}[verbal product]
    Let \(\{\G_i\}_{i\in I}\) be a countable family of groups.\ Denote by \([\G_i]^{\mathcal{F}}\) the normal closure in the free product \(\mathcal{F}\coloneqq\bigast_{i\in I}\G_i\) of the set \[\big\{[\G_i,\G_j]\big\vert i\neq j \in I\big\}.\]
    Let \(W\subseteq \mathbb{F}_{\infty}\) be a set of words and \(W(\mathcal{F})\) the corresponding verbal subgroup of \(\mathcal{F}\).\
    The \emph{verbal product} of \(\{\G_i\}_{i\in I}\) with respect to \(W\) is the quotient group
    \[\bigast_{i\in I}^{_W}\G_i\coloneqq {\frac{\mathcal{F}}{[\G_i]^{\mathcal{F}}\cap W(\mathcal{F})}} \ \ \Bigg(={\frac{\bigast_{i\in I}\G_i}{[\G_i]^{\mathcal{F}}\cap W(\bigast_{i\in I}\G_i)}}\Bigg). \]
    
\end{df}
 
\begin{exa}
Let \(A\) and \(B\) be groups, 
    \begin{enumerate}
        \item If \(W=\{\}\) is the empty word, then the verbal subgroup \(W(A\ast B)\) associated to \(W\) is the identity of \(A\ast B\).\ The verbal product \(A \ast^{_W} B\) coincides with the free product \(A \ast B\).
        \item Let \(n_1\coloneqq [x_2, x_1]\).\ If \(W = \{n_1\}\) the verbal subgroup associated is the commutator subgroup.\ Then the verbal product gives the direct product of groups \(A\ast^{_W}B = A \times B\).
        \item The words \(n_k 
        \coloneqq[x_{k+1}, n_{k-1}]\) with \(k \in \N_{\geq2}\) recursively yield the lower central series as associated verbal subgroups.\
        If \(W=\{n_k\}\), then the verbal product \(A \ast^{_W} B\) was introduced by Golovin in 1950 \cite{Gol} and it is called \(k\)-\emph{nilpotent product}.\
        \item The words
    \[s_1(x_1, x_2) \coloneqq [x_1, x_2],\]
    \[s_k(x_1,\dots , x_{2^k}) \coloneqq [s_{k-1}(x_1,\dots, x_{2^{k-1}}), s_{k-1}(x_{2^{k-1}+1}, \dots , x_{2^k})],\] recursively yield the derived series as associated verbal subgroup. When \(W=\{s_k\}\), the verbal product \(A \ast^{_W} B\) is called the \(k\)-\emph{solvable product}.
        \item The word \(x^k\) gives the \(k\)-Burnside verbal subgroup, namely the group generated by the \(k^{th}\) power of the elements of the group. If \(W = \{x^k\}\), the verbal product \(A\ast^{_W}B\) is called the \emph{k-Burnside product}.
    \end{enumerate}
\end{exa}

\begin{rk}[projection onto the direct sum]
    Let \(\{\G_i\}_{i\in I}\) be a countable family of groups, and let \(W\) be a set of words in \(\mathbb{F}_{\infty}\).\ The verbal product of the \(\G_i\)'s projects onto the direct sum.\ Indeed the direct sum of countable groups can be defined as the quotient of the free product \(\F\coloneqq \bigast_i\G_i\) by the normal closure of the commutators.\ In symbols \[\bigoplus_{i\in I}\G_i=\frac{\F}{[\G_i]^{\F}}.\] The group homomorphism 
    \begin{eqnarray*}
         p\colon  \frac{\F}{[\G_i]^{\F}\cap W(\F)} &\to& \frac{\F}{[\G_i]^{\F}} 
    \end{eqnarray*}
    that sends the equivalence class of an element \(\gamma\in \F\) to the coset of \(\gamma\) in the codomain, is well defined and surjective as \([\G_i]^{\F}\cap W(\F)\subseteq [\G_i]^{\F}\).\
    
\end{rk}

Here we are interested in verbal wreath products that have been recently extensively studied by Brude and Sasyk \cite{Brude}:

\begin{df}[verbal wreath product]\label{wrw}
    Let \(\G\) and \(A\) be countable groups.\ Let \(W \subseteq \mathbb{F}_{\infty}\) be a set of words and let \(X\) be a countable \(A\)-set.\ Let us denote by \(\F\), the free product of \(|X|\)-copies of \(\G\).\
    There is an action \(\alpha\colon A \curvearrowright \F\) , given by permuting the copies of \(\G\). That
is, if \((g)_{x_1}\) denotes the element \(g\) in the copy \(x_1\) of \(\G\) in \(\F\), then 
    \[\alpha(a)((g)_{x_1}) = (g)_{a\cdot x_1}.\]
Since \(\alpha(a)\) is an automorphism for every \(a\in A\), \([\G_x]^{\F}\) is \(A\)-invariant and \(W(\F)\) is fully invariant, we have that \(W (\F) \cap [\G_x]^{\F}\) is invariant under the action of \(A\).\ Hence, there is a well-defined action \(A \curvearrowright \bigast_X^{_W}\G\). 
The semi-direct
product 
\[\G\wr^{_W}_XH\coloneqq \bigast_{x\in X}^{_W} \G\rtimes_{\alpha} A\]
is called \emph{restricted permutational verbal wreath product} of \(\G\) and \(A\).
\end{df}

We call \(k\)-nilpotent, \(k\)-solvable and \(k\)-Burnside wreath products the restricted permutational wreath products associated to \(k\)-nilpotent, \(k\)-solvable and \(k\)-Burnside products, respectively.

\begin{prop}\label{amker}
    In the setting of Definition \ref{wrw}, the kernel of the projection 
    \[p:\G\wr_X^{_W}A\to \G\wr_XA\]
    is amenable whenever \(W\) is solvable, nilpotent or \(k\)-Burnside for \(k=2,3,4,6\). 
\end{prop}
\begin{proof}
    Note that the kernel of the projection \(p\) is the subgroup \[[\G_x]^{_W}\coloneqq \frac{[\G_x]^{\F}}{[\G_x]^{\F}\cap W(\F)}.\] As proved by Brude and Sasyk \cite[Theorem 2.16 and Corollary 2.17]{Brude} we have that the quotient \([\G_x]^{_W}\) is isomorphic to a subgroup of the verbal product 
    \[\bigast_{x\in X}^{_W}\frac{\G_x}{W(\G_x)}.\]
    Now if \(W\) is nilpotent, solvable or \(k\)-Burnside for \(k=2,3,4,6\), then the quotient \(\frac{\G}{W(\G)}\) is a nilpotent, solvable or \(k\)-Burnside group, respectively.\ In particular \(\frac{\G}{W(\G)}\) is amenable (\(k\)-Burnside groups are known to be finite when \(k=2,3,4,6\) \cite{Burnside}).\ Since countable verbal products of amenable groups are amenable when the verbal product is nilpotent, solvable or \(k\)-Burnside for \(k = 2, 3, 4, 6\) \cite[Corollary 1.6]{Brude}, the kernel of the projection is amenable. 
\end{proof}

In the theory of bounded cohomology \emph{amenable morphisms} are important:

\begin{df}[amenable group homomorphism]
    A surjective homomorphism \(\phi\colon G\to H\) of groups with kernel
\(K\) is \emph{amenable} if the inflation map
\[H^{\bullet}_b(\phi;I_V)\colon H_b^{\bullet}(H; V^K) \to H_b^{\bullet}(G; V)\]
is an isometric isomorphism for all dual normed \(G\)-modules \(V\). 
\end{df}

\begin{cor}[Corollary \ref{B}]\label{verbalcase}
    Let \(W\subseteq \mathbb{F}_{\infty}\) be nilpotent, solvable or \(k\)-Burnside for \(k=2,3,4,6\).\ Then the projection map \(p:\G\wr_X^{_W}A\to \G\wr_XA\) is an amenable homomorphism.\ In particular if \(A\) is \X{s}-boundedly \(n\)-acyclic and the \(A\)-orbits of elements of \(X\) are infinite, then also \(\G\wr_X^{_W}A \) is \X{s}-boundedly \(n\)-acyclic.
\end{cor}
\begin{proof}
    By Proposition \ref{amker}, the kernel of the projection \(p\colon \G\wr_X^{_W}A\to \G\wr_XA\) is amenable. A surjective group homomorphism with amenable kernel is an amenable morphism by Gromov's mapping theorem for coefficients \cite[Section 3.1]{Grom82} \cite[Theorem 3.1.4]{m.rap}.\ Now whenever \(A\) is \X{s}-boundedly \(n\)-acyclic the permutational wreath product of \(\G\wr_XA\) is \X{s}-boundedly \(n\)-acyclic (Theorem \ref{A}).\ Since the submodule of invariants of a semi-separable coefficient module is semi-separable for the quotient (Lemma \ref{ssep inv}), using the fact that the projection is an amenable morphism, we obtain \(H_b^p(\G\wr_X^{_W}A;V)=0\) for every semi-separable coefficient module \(V\).
\end{proof}

\section{Application to stable commutator length}\label{7}

\emph{Stable commutator length} (\(\scl\)) is a real-valued invariant of groups
that allows for algebraic, topological, and analytic descriptions \cite{Cal}.\ 
In this section, we prove Theorem \ref{C} of the introduction, which shows the vanishing of \(\scl\) for a large class of permutational verbal wreath products.

\begin{df}[(stable) commutator length]
Given a group \(\G\), the \emph{commutator length} of an element
\(g\) in the commutator subgroup \([\G,\G]\), denoted as \(\cl(g)\), is the least number of commutators in \(\G\) whose product is \(g\).\ 
The \emph{stable commutator length} of \(g\), denoted \(\scl(g)\), is the limit \[\scl(g) \coloneqq \lim_{n\to \infty} \frac{\cl(g^n)}{n}.\] 
\end{df}
Stable commutator length is closely related to quasimorphisms (and second bounded cohomology) via Bavard's Duality Theorem \cite[Theorem 2.70]{Cal}.\ Calegari showed that groups satisfying a law have vanishing stable commutator length \cite{cal2}.\ We are going to use this fact to prove the vanishing of stable commutator length for all verbal wreath products of discrete groups with \X{s}-boundedly \(2\)-acyclic groups.
\begin{df}[law]
    A group \(\G\) obeys a \emph{law} if there is a free group \(F\) and a non-trivial element \(\omega \in F\) such that for every homomorphism
\(\rho\colon F \to \G\), we have \(\rho(\omega)= id\).
\end{df}

\begin{exa}
    Every Abelian group satisfies the law \(xyx^{-1}y^{-1}\), on the other hand, free groups do not satisfy any non-trivial law.
\end{exa}

\begin{rk}\label{prodlaw}
    If the groups \(\G\) and \(H\) satisfy a law then the extension \[1\to \G\to \G'\to H\to 1\] also satisfies a law.\
    Indeed if \(v\) is a law for  \(H\) and \(w\) is a law for \(\G\), then the valuation \(w(v)\), of \(w\) in \(v\) is a law for the extension \(\G'\).\ Moreover, the word \(w(v)\) is non trivial.\ This result comes from the theory of varieties of groups and in particular from the fact that if a product variety is the trivial variety of all groups then one term must be the trivial variety \cite[Chapter 2, 21.21]{HNeu}.
\end{rk}

\begin{df}[(homogeneous) quasimorphism]
    Let \(\G\) be a group, a map \(f \colon\G\to\R\) is a \emph{quasimorphism} if there exists a
constant \(D \geq 0\) such that
\[|f(g_1) + f(g_2) - f(g_1g_2)| \leq D\]
for every \(g_1, g_2 \in \G\).\ The least \(D \geq 0\) for which the above inequality is
satisfied is called the \emph{defect} of \(f\).\\
A quasimorphism \(f \colon \G \to \R\) is \emph{homogeneous} if \(f(g^n) =
n f(g)\) for every \(g \in\G\) and for all \(n\in \Z\). The space of homogeneous quasimorphisms is denoted by \(Q^h(\G; \R)\).
\end{df}

We rewrite Calegari's result on stable commutator length of groups admitting a law in terms of second bounded cohomology as follows:

\begin{prop}\label{law0}
    Let \(\G\) be a group that satisfies a non-trivial law \(\omega\), then we have~\(H_b^2(\G;\R)=0\).
\end{prop}
\begin{proof}

    Let \([\sigma]\) be a bounded class in \(H^2(\G;\R)\).\ Then there exists a  central extension 
 \[1\to \R\to \G'\xrightarrow{p} \G \to1 \] associated to \(\sigma\) \cite[Lemma 2.1]{fri}.\ Moreover, it holds \(H^2(p)([\sigma])=0\in H^2(\G';\R)\) \cite[Lemma 2.3]{fri}.\ 
 This implies that the map \[H^2_b(p)\colon H_b^2(\G;\R)\to H_b^2(\G';\R)\] sends \([\sigma]\) to an element of the kernel of the comparison map of \(\G'\), i.e.\ \(H_b^2(p)([\sigma])\in EH_b^2(\G';\R)=\ker c^2_{\G'}\).\ 
 Recall that \cite[Corollary 2.11]{fri} \[EH_b^2(\G';\R)\cong \frac{Q^h(\G';\R)}{\text{Hom}(\G';\R)}.\]
 Now since both \(\G\) and \(\R\) satisfy a non-trivial law, by Remark \ref{prodlaw}
 also the extension \(\G'\) satisfies a non trivial law.\
 By Calegari's theorem, we know that groups satisfying a non-trivial law have vanishing stable commutator length, so in particular \(\scl_{\G'}(g)=0\) for every \(g\in [\G',\G']\) \cite{cal2}.\ 
 By Bavard's Duality Theorem the following equality holds \cite[Theorem 2.70]{Cal}
 \[\scl_{\G}(g)=\frac{1}{2}\sup\Big\{\frac{|\phi(g)|}{D(\phi)}\;\Big\vert \;\phi\in \frac{Q^h(\G';\R)}{\text{Hom}(\G';\R)}\Big\}.\]
 In particular the vanishing of \(\scl_{\G'}\) implies that every homogeneous quasimorphism in \(\G'\) is a homomorphism and hence \(EH^2_b(\G';\R)\cong0\).\ We have shown that the image of \([\sigma]\) under \(H^2_b(p)\) is trivial.\ However, every epimorphism induces an injective map in second bounded cohomology \cite[Theorem 2.17]{fri}. This implies that \([\sigma]~=~0\).
\end{proof}

\begin{teo}[Theorem \ref{C}]\label{sclcase}
    Let \(A\) be a countable group, and let \(X\) be an \(A\)-set such that all the \(A\)-orbits are infinite. 
    Let \(\G\) be any group and let \( W\subseteq \mathbb{F}_{\infty}\) be any non-empty set.\
    If \(A\) is \X{s}-boundedly \(2\)-acyclic, then the stable commutator length of the permutational verbal wreath product \(\G\wr_X^{_W}A\) vanishes. 
\end{teo}
\begin{proof}
    Consider the projection 
    \(p:\G\wr_X^{_W}A\to \G\wr_XA.\) 
    Its kernel is the quotient
    \[[\G_x]^{_W}\coloneqq \frac{[\G_x]^{\F}}{[\G_x]^{\F}\cap W(\F)}.\]
     If \(W\) is non empty, then \([\G_x]^{_W}\) satisfies a non-trivial law and by Proposition \ref{law0}, we know that \[H_b^2([\G_x]^{_W};\R)=0.\] 
     Since the map \(p\) is surjective and has boundedly \(2\)-acyclic kernel, the map
     \[H^p_b(p;\R)\colon H_b^p(\G\wr_XA; \R ) \to H_b^p(\G\wr_X^{_W}A;\R )\]
is an isomorphism for \(p=2\) and injective for \(p=3\) \cite[Theorem 4.1.1]{m.rap}.
In particular whenever \(A\) is \X{s}-boundedly \(2\)-acyclic we obtain \(H_b^2(\G\wr_XA; \R )=0\). In particular the kernel of the comparison map will also be equal to 0: 
\[EH_b^2(\G\wr^{_W}_XA; \R )=0.\] 
This implies that every homogeneous quasimorphism is a group homomorphism \cite[Corollary 2.11]{fri}.
By Bavard's Duality Theorem a group that has no non-trivial homogeneous quasimorphisms has vanishing stable commutator length \cite[Theorem 2.70]{Cal}.
\end{proof}

\section{Application to embeddings}\label{8}

In a recent article by Wu, Wu, Zhao, and Zhou \cite{embed2}, the authors showed that every group of type \(F_n\) embeds quasi-isometrically into a boundedly acyclic group of type \(F_n\) as an offspring of the study of labeled Thompson groups.\ This result extends a previous embedding theorem by Fournier-Facio, L\"oh and Moraschini \cite[Theorem 2]{embed1}.\
In this section we prove Theorem \ref{D} of the introduction in which we use Theorem \ref{A} and the properties of Thompson's group \(F\) to generalize points 1, 4 and 5 of Wu, Wu, Zhao, and Zhou's embedding result \cite[Theorem 0.1]{embed2}.\medskip

We recall that, for any group \(\G\) and for every \(n\in \N\) we say that \(\G\) is of type \(F_n\) if it admits a classifying space with a finite \(n\)-skeleton.\
The group \(\G\) is of type \(FP_n\) if the \(\G\)-module \(\Z\) admits
a projective resolution which is finitely generated in all dimensions \(\leq n\).

\begin{teo}[Theorem \ref{D}]\label{finitenesscase}
    Let \(\G\) be a group of type \(FP_m\).\ Then \(\G\) injects isometrically into a  \X{s}-boundedly acyclic group of type \(FP_m\).
\end{teo}
\begin{proof}
    Let us consider Thompson's group \(F\) and the set \(X\coloneqq(0,1)\cap \Z[1/2].\) The group \(F\) acts diagonally on \(X^i\) with finitely many orbits for every \(i\in \Z\).\ Moreover the stabilizers of this action are of type \(F_{\infty}\).\ Indeed for every \(x\in X\), the stabilizer \(F_x\) consists in the set of maps in \(F\) that fix the dyadic point \(x\).\ This group can be identified with the product \(F\times F\),
     which is of type \(F_{\infty}\) \cite[Proposition 1.4.4]{belk}.\ In this setting a result by Bartoldi, de Cornulier and Kochlouva shows that the wreath product \(\G\wr_XF\) is also of type \(FP_m\) \cite[Proposition 5.1]{fini}.\ Moreover, as seen in Example \ref{Thompson F}, the wreath product \(\G\wr_XF\) is \X{s}-boundedly acyclic.\ The isometric embedding \[\G\hookrightarrow \G\wr_XF\] concludes the proof.
\end{proof}

The same result holds for groups of type \(F_m\):

\begin{cor}
    Let \(\G\) be a group of type \(F_m\).\ Then \(\G\) injects isometrically into a  \X{s}-boundedly acyclic of type \(F_m\).
\end{cor}
\begin{proof}
    If \(m=1\) the statement is trivial as wreath product of finitely generated groups with finitely many orbits is finitely generated.\ If \(m\geq 2\) we use the fact that a group is of type \(F_m\) if and only if it is finitely generated and of type \(FP_m\).\ With the notation used in the proof of Theorem \ref{finitenesscase}, if \(\G\) is finitely presented, then \(\G\wr_XF\) is also finitely presented \cite[Theorem 1.1]{finpres}.\ Using Theorem \ref{finitenesscase} we get that \(\G\wr_XF\) is of type \(FP_m\), hence the statement.
\end{proof}
\begin{rk}
The above results hold for any \(FP_m\) and \(F_m\) with \(m\in \N\cup \{\infty\}\).\ As pointed out to me by Xiaolei Wu, the cases of \(FP_m\) and \(F_m\) with finite \(m\) could be obtained using Monod's result \cite[Theorem 3]{lamp} and the Houghton group \(\mathcal{H}_{m+1}\) in place of \(F\).   
\end{rk}

\begin{rk}
    If the group \(\G\) has infinite abelianization then the wreath product \(\G\wr_XF\) is of type \(FP_m\) (respectively \(F_m\)) if and only if \(\G\) is of type \(FP_m\) (respectively \(F_m\)) \cite[Theorem A]{fini}.
\end{rk}

\bibliographystyle{alpha}
\bibliography{ref_project}

\end{document}